\documentclass[12pt]{article}

\usepackage{mystyle}
\usepackage{adjustbox}

\newcommand{\cat}			{\mathcal{C}}
\newcommand{\dat}			{\mathcal{D}}

\newcommand\define[1]	       {{\bf{#1}}}

\newcommand{\mono}            {{\hookrightarrow}}

\newcommand{\uExt}			{\underline{\mathrm{Ext}}}
\newcommand{\uIExt}			{\underline{{\mathcal{E}\kern -.5pt \mathrm{xt}}}}

\renewcommand{\lim}            {{\mathsf{lim}}\,}
\newcommand{\colim}            {{\mathsf{colim}}\,}

\newcommand{\im}               {{\mathrm{im}}\,}
\renewcommand{\ker}            {{\mathrm{ker}}\,}
\newcommand{\coker}            {{\mathrm{coker}}\,}
\newcommand{\id}               {{\mathrm{id}}}

\newcommand{\Int}		{\mathrm{Int}\,}

\newcommand{\Inc}               {{\mathrm{Inc}}}
\newcommand{\Z}                 {{\mathbb{Z}}}

\newcommand{\One}             {{\mathbf{1}}}

\newcommand{\Gr}                {\mathcal{R}}

\newcommand{\op}                {\mathrm{op}}

\newcommand{\IHom}{{\mathcal{H}\kern -.5pt \mathrm{om}}}
\newcommand{\Hom}{\mathrm{Hom}}
\newcommand{\uHom}{\underline{\mathrm{Hom}}}
\newcommand{\uIHom}{\underline{{\mathcal{H}\kern -.5pt \mathrm{om}}}}
\newcommand{\Nat}{\mathrm{Nat}}

\title{A Categorical Approach to Möbius Inversion via Derived Functors}
\author[1]{Alex Elchesen}
\author[1]{Amit Patel}
\affil[1]{Department of Mathematics, Colorado State University}
\date{}

\begin{document}
	
\maketitle

\abstract{
We develop a cohomological approach to M\"obius inversion using derived functors in the enriched categorical setting. For a poset $P$ and a closed symmetric monoidal abelian category~$\cat$, we define M\"obius cohomology as the derived functors of an enriched hom functor on the category of $P$-modules. We prove that the Euler characteristic of our cohomology theory recovers the classical M\"obius inversion, providing a natural categorification. As a key application, we prove a categorical version of Rota's Galois Connection. Our approach unifies classical ideas from combinatorics with homological algebra.}

\section{Introduction} \label{sec:intro}

M\"obius inversion is a fundamental principle in combinatorics, generalizing the classical inclusion-exclusion principle to partially ordered sets. 
Originally emerging from number theory in the 19th century, 
the concept was later vastly generalized by Gian-Carlo Rota~\cite{rota1964foundations}, who recognized its deeper structural role in combinatorics. 
At its core, M\"obius inversion provides a systematic way to recover local information from global summaries. Given a function $f : P \to \Gr$ from a finite poset $P$ to a commutative ring $\Gr$, 
its upper M\"obius inversion is the unique function $\partial_+ f : P \to \Gr$ such that for all $a \in P$,
\[
f(a) = \sum_{b : a \leq b} \partial_+ f(b).
\]
Dually, its lower M\"obius inversion is the unique function $\partial_- f : P \to \Gr$ such that for all $b \in P$,
\[
f(b) = \sum_{a : a \leq b} \partial_- f(a).
\]

Our entry into M\"obius inversion stems from its fundamental role in persistent homology, a key tool in topological data analysis~\cite{Patel:2018, persistence_over_posets}. 
In pursuit of generalizations, Patel and Skraba developed M\"obius homology for poset modules $M : P \to \cat$ (functors valued in an abelian category $\cat$)~\cite{patel2023mobius_homology}.
Their theory associates a homology object $H_d^\downarrow M(a)$ to every $a \in P$ and index $d$ such that
its Euler characteristic recovers the lower M\"obius inversion of the dimension function
$m : P \to \Gr(\cat)$. They construct these homology objects via a simplicial cosheaf over the order complex of $P$.

In this paper, we develop a cohomological approach to M\"obius inversion through derived functors. Rather than dualizing the construction of Patel and Skraba to simplicial sheaves, we provide a more general and functorial perspective through enriched hom functors. Specifically, given a point $a\in P$, we define the indicator module $\One_a$. When $\cat$ is closed monoidal, we show that the set of natural transformations $\Nat(\One_a,M)$ forms an object $\uHom(\One_a,M)$ of $\cat$ (Definition \ref{def:enriched_hom}). The functor $\uHom(\One_a,-):\cat^P\to \cat$ is left-exact, and we define its right-derived functors $\uExt^d(\One_a,M)$ as our M\"obius cohomology objects.

This derived functor approach offers several key advantages. First, it generalizes naturally to handle cohomology for arbitrary spreads and modules, extending beyond the traditional scope of M\"obius inversion. Second, it reveals connections with classical homological algebra, allowing us to leverage established theoretical machinery. Our framework culminates in a categorical version of Rota's Galois Connection Theorem, expressed as an enriched adjunction.

\paragraph{Previous Work}

The study of the total cohomology of $P$-modules using derived functors has a rich history, with foundational contributions by Deheuvels~\cite{Deheuvels1962}. Baclawski made significant advances in understanding the homological aspects of posets through two key papers. His 1975 work~\cite{Baclawski1975} developed a homological interpretation of Whitney numbers of geometric lattices, while his 1977 paper~\cite{Baclawski1977} established deep connections between Galois connections and spectral sequences. While this body of work laid important groundwork, the explicit connection to M\"obius inversions remained unexplored until the recent work of Patel and Skraba~\cite{patel2023mobius_homology}.

A parallel development in the categorical treatment of M\"obius inversion emerged through the work of Leroux~\cite{Leroux1975}, who introduced the concept of M\"obius categories. This framework was significantly expanded by Content, Lemay, and Leroux~\cite{Content1980}, who developed a comprehensive categorical setting for M\"obius inversion. While their work established important functorial properties of incidence algebras and M\"obius functions, it did not pursue the homological aspects that we develop here.

Our investigation originates from questions in persistent homology, a fundamental tool in topological data analysis. While we do not explicitly address persistent homology in this paper, the discussion of persistent homology in~\cite{patel2023mobius_homology} dualizes to our setting.
In this context, our M\"obius cohomology theory shares important connections with recent work by Oudot and Scoccola~\cite{doi:10.1137/22M1489150} on bigraded Betti numbers.

Recent work by G\"ulen and McCleary~\cite{GulenMcCleary} demonstrated how Rota's Galois connection theorem reveals the functorial character of M\"obius inversion. Our Theorem~\ref{thm:categorical_RGCT} substantially generalizes their observation, providing a comprehensive categorical framework that unifies the combinatorial and functorial perspectives on M\"obius inversion.

\paragraph{Our contributions}

The main contributions of this paper are:
\begin{enumerate}
    \item We develop a cohomological theory for M\"obius inversion using derived functors in the enriched categorical setting (Section~\ref{sec:mobius_cohomology}). We prove that the Euler characteristic of our cohomology theory coincides with the classical M\"obius inversion (Theorem~\ref{thm:cohom_euler_char}), providing a natural categorification while simultaneously offering a more general framework for studying poset modules.
    
    \item We establish an explicit formula for computing M\"obius cohomology using the standard cofree resolution (Section~\ref{sec:derived_hom}). Specifically, we show that applying $\uHom(\One_a,-)$ to the standard cofree resolution (Equation~\ref{eq:cofre_res}) yields a cochain complex (Equation~\ref{eq:standard_mobius_complex}) whose cohomology groups recover M\"obius cohomology. This formula reveals how our derived functor approach recovers the construction using sheaves on the order complex.

    \item We prove a categorical version of Rota's Galois Connection Theorem (Theorem~\ref{thm:categorical_RGCT}) that implies the classical result (Theorem~\ref{thm:rota_classic}). This generalization takes the form of an enriched adjunction between the pushforward and pullback functors induced by a Galois connection (Section~\ref{sec:Galois_connections}).
    
    \item We demonstrate that our derived functor approach unifies various perspectives on M\"obius inversion through several key isomorphisms (Proposition~\ref{prop:cofree_hom}, Proposition~\ref{prop:euler_indicator}, and Theorem~\ref{thm:rota_ext}), connecting classical combinatorial methods, homological algebra, and categorical techniques in a single coherent framework.
\end{enumerate}

Throughout, we have made an effort to keep the exposition self-contained, providing background material and motivation for key concepts in Section~\ref{sec:preliminaries}. This makes the paper accessible to readers with a basic understanding of category theory and homological algebra.

\paragraph{Outline}

The paper is organized as follows. Section~\ref{sec:preliminaries} establishes the categorical preliminaries, including key concepts from homological algebra and enriched category theory. Section~\ref{sec:poset_modules} introduces poset modules and develops the fundamental constructions needed for our cohomology theory. Section~\ref{sec:derived_hom} defines derived hom functors and explores their properties in the context of poset modules. Section~\ref{sec:mobius_cohomology} introduces M\"obius cohomology and establishes its relationship to classical M\"obius inversion. Finally, Section~\ref{sec:Galois_connections} examines Galois connections between posets and proves our categorical version of Rota's theorem.

\section{Categorical Preliminaries}
\label{sec:preliminaries}

In this section, we introduce the key categorical concepts, constructions, and notation used throughout the paper. These preliminaries provide the necessary framework for later sections, focusing on abelian and monoidal categories, ends, and basic homological algebra. For a deeper study, see \cite{maclane1998categories,weibel1994homological}.

\subsection{Adjunctions}

For a category $\cat$, we denote by $\Hom_\cat (A,B)$ the set of morphisms from $A$ to $B$. 
Recall that an adjunction between two categories $\cat$ and $\dat$ consists of a pair of functors 
$F : \cat \to \dat$ and $G : \dat \to \cat$, together with natural isomorphisms of sets 
$\Phi : \Hom_\dat(F(A), B) \cong \Hom_\cat(A, G(B))$. 
In this situation, we say $F$ is \emph{left adjoint} to $G$ and $G$ is \emph{right adjoint} to $F$, denoted $F \dashv G$.
An important property of left adjoints is that they are cocontinuous, while right adjoints are continuous, 
meaning left adjoints preserve colimits and right adjoints preserve limits.

\subsection{Abelian Categories}

A category $\cat$ is \emph{preadditive} if for every set $\Hom_\cat(A,B)$ there is an abelian group structure, 
and composition is bilinear, i.e., 
\[ f \circ (g + h) = (f \circ g) + (f \circ h), \quad (f + g) \circ h = (f \circ h) + (g \circ h). \]
A functor $F : \cat \to \dat$ between preadditive categories is \emph{additive} if for all objects $A, B$, 
the induced map 
$\Hom_\cat(A, B) \to \Hom_\dat(F(A), F(B))$
is a group homomorphism.

A \emph{biproduct} of objects $A_1, \ldots, A_n$ in a preadditive category $\cat$ is an object 
$A_1 \oplus \cdots \oplus A_n$ with morphisms $\pi_k : A_1 \oplus \cdots \oplus A_n \to A_k$ 
and $\iota_k : A_k \to A_1 \oplus \cdots \oplus A_n$ satisfying:
\[
\pi_k \circ \iota_k = \id_{A_k}, \quad \pi_l \circ \iota_k = 0 \text{ for } k \neq l.
\]
In this case, $\big( A_1 \oplus \cdots \oplus A_n, \pi_k \big)$ is a product, and 
$\big( A_1 \oplus \cdots \oplus A_n, \iota_k \big)$ is a coproduct. 
A preadditive category is \emph{additive} if it admits all finite biproducts.

A preadditive category $\cat$ is \emph{abelian} if it has a zero object, all binary biproducts, 
kernels and cokernels, and every monomorphism and epimorphism is normal. 
Functors between abelian categories are typically additive, and if $F \dashv G$, both $F$ and $G$ are automatically additive.

Exactness is central in abelian categories. An additive functor $F : \cat \to \dat$ between abelian categories is:
	\begin{itemize}
	\item \emph{left-exact} if it preserves exactness of sequences of the form 
  \[ 0 \to A \to B \to C, \]
  \item \emph{right-exact} if it preserves exactness of sequences of the form 
  \[ A \to B \to C \to 0. \]
	\end{itemize}

A functor is \emph{exact} if it is both left-exact and right-exact. If $F \dashv G$, then $F$ is automatically right-exact and $G$ is left-exact.

\subsection{Closed Monoidal Categories}

A \emph{monoidal category} consists of a category $\cat$, a bifunctor $\otimes : \cat \times \cat \to \cat$ 
(called the tensor product), an object $\One$ (the unit object), and natural isomorphisms:
\[
\alpha_{A,B,C} : (A \otimes B) \otimes C \cong A \otimes (B \otimes C), \quad 
\lambda_A : \One \otimes A \cong A, \quad \rho_A : A \otimes \One \cong A.
\]
These isomorphisms satisfy certain coherence conditions (see \cite[VII]{maclane1998categories}). 
A monoidal category is \emph{symmetric} if there is a natural isomorphism $s_{A,B} : A \otimes B \cong B \otimes A$ 
that satisfies its own coherence conditions.
A (symmetric) monoidal category $(\cat, \otimes, \One)$ is \emph{closed} if for each $A \in \cat$, the functor 
$-\otimes A: \cat \to \cat$ has a right adjoint. This right adjoint is denoted $\IHom_\cat(A, -) : \cat \to \cat$ 
and is called the \emph{internal hom}.
\begin{prop}
\label{prop:hom_one}
For any closed symmetric monoidal category $(\cat, \otimes, \One)$ and for all objects $B$ in $\cat$, we have
$\IHom_\cat(\One, B) \cong B.$
\end{prop}
\begin{proof}
For any $A\in \cat$, we have $\Hom_\cat(A,\IHom_\cat(\One,B)) \cong \Hom_\cat(A\otimes \One, B)\cong \Hom_\cat(A,B)$. The result follows now from the Yoneda lemma.
\end{proof}

\subsection{Abelian Symmetric Monoidal Categories}

We will focus on categories that are both abelian and closed symmetric monoidal. A key compatibility condition between the abelian and monoidal structures 
is that the tensor product must be additive in both variables. 
When the monoidal structure is closed, this compatibility condition is automatically satisfied because biproducts serve as both limits and colimits, and the tensor product is left adjoint to the internal hom functor.
\begin{center}
\textbf{Henceforth, $\cat$ is a small, complete, abelian, and closed symmetric monoidal category.}
\end{center}

\subsection{Ends}

Fix a functor $S : \cat^\op \times \cat \to \dat$. A \emph{wedge} of $S$, denoted $e: W \to S$, is an object $W \in \dat$ together with maps 
$e_A : W \to S(A,A)$ such that for every $f: A \to B$, the diagram
\begin{equation*}
		\begin{tikzcd}
			W \ar[rr, "e_{B}"] \ar[d, "e_{A}"] && S(B,B) \ar{d}{S(B,f)}\\
			S(A,A) \ar{rr}{S(f,B)} && S(B,A)
		\end{tikzcd}
	\end{equation*}
commutes. The \emph{end} of $S$, denoted $\int_A S(A,A)$, is the universal wedge, meaning any wedge $e: W \to S$ 
factors uniquely through it.
That is, if $e:W\to S$ is any wedge of $S$, then there is a unique map $\phi:W\to \int_A S(A,A)$ 
satisfying $\pi_A \circ \phi = e_A$, for all objects~$A \in \cat$.

The end of an ordinary functor $T:\cat\to\dat$, denoted $\int_A T(A)$, is defined by taking the end of the 
composition $S = \cat^\op\times \cat\stackrel{p}{\to}\cat\stackrel{T}{\to}\dat$, 
where $p$ denotes the projection functor onto the second factor. 
In this case, the end reduces to the limit of $T$.

Given functors $F,G:\cat\to\dat$, the set of natural transformations $\Nat(F,G)$ from $F$ to $G$ is given (up to isomorphism) by the formula
\begin{equation}\label{eq:end_formula_natural_transformations}
	\Nat(F,G) \cong \int_{A}\Hom_\dat(F(A),G(A)).\end{equation}
This formula and generalizations of it are central to our constructions. If $\dat$ is a closed category, then we can replace $\Hom_\dat(F(A),G(A))$ in the formula above by $\IHom_\dat(F(A),G(A))$ to make $\Nat(F,G)$ an object of $\dat$ (see Definition~\ref{def:enriched_hom} below).

\subsection{Injective Resolutions}
 \label{subsec:injective_resolutions}

An object $I$ in $\cat$ is called \emph{injective} if, for every monomorphism $f : A \hookrightarrow B$ and any morphism $g : A \to I$, there exists a morphism $h : B \to I$ making the following diagram commute:
\begin{equation*}
\begin{tikzcd}
A \ar[r, hookrightarrow, "f"] \ar[d, "g"] & B \ar[dl, dashrightarrow, "h"] \\
I &
\end{tikzcd}
\end{equation*}

Arbitrary products of injective objects are injective. Moreover, in an abelian category, finite direct sums (i.e., biproducts) of injective objects remain injective since finite sums coincide with products. However, arbitrary direct sums do not necessarily preserve injectivity.

An abelian category $\cat$ is said to have \emph{enough injectives} if every object $A$ admits a monomorphism $A \mono I$ into an injective object $I$.

An \emph{injective resolution} of an object $X$ in $\cat$ is an exact sequence
\[
0 \to X \hookrightarrow I^0 \to I^1 \to I^2 \to \cdots
\]
where each $I^d$ is an injective object in $\cat$.

If $\cat$ has enough injectives, every object $X$ admits an injective resolution via the standard construction. Begin by embedding $X$ into an injective object $I_0$, and successively embed each cokernel into a new injective object $I_{i+1}$ as follows:
\[
\begin{tikzcd}[column sep = 2em]
X \ar[r, hookrightarrow, "\epsilon_0"]  & I_0 \ar[r, twoheadrightarrow] & \coker \epsilon_0 \ar[r, hookrightarrow, "\epsilon_1"] & I_1 \ar[r, twoheadrightarrow] & \coker \epsilon_1 \ar[r, hookrightarrow, "\epsilon_2"] & I_2 \ar[r, twoheadrightarrow] & \coker \epsilon_2 \ar[r, hookrightarrow, "\epsilon_3"] & \cdots
\end{tikzcd}
\]

\subsection{Grothendieck Ring}
\label{subsec:Grothendieck}

The \emph{Grothendieck ring} of a category $\mathcal{C}$, denoted $\Gr(\mathcal{C})$, is the free abelian group generated by the isomorphism classes $[A]$ of objects $A$ in $\mathcal{C}$, subject to the relation $[B] = [A] + [C]$ for every short exact sequence $0 \to A \to B \to C \to 0$ in $\mathcal{C}$. The ring structure on $\Gr(\mathcal{C})$ is induced by the monoidal structure of $\mathcal{C}$, where the product is given by $[A] \cdot [B] = [A \otimes B]$, with $\otimes$ denoting the tensor product in $\mathcal{C}$. The resulting ring $\Gr(\mathcal{C})$ is commutative and unital, with the class of the unit object~$[\mathbf{1}]$ acting as the multiplicative identity: $[A] \cdot [\mathbf{1}] = [A]$. The multiplication is both commutative and associative, reflecting the symmetric monoidal structure of $\mathcal{C}$.

\section{Poset Modules}
 \label{sec:poset_modules}
In this section, we introduce poset modules, which are functors from a poset (viewed as a category) to the target category $\cat$. We establish key definitions and notation that will be used throughout the paper, focusing on constructions such as cofree modules, enriched hom-objects, tensor products, and injective resolutions.

\subsection{First Constructions}
We begin by defining posets and their interpretation as categories, along with the notion of $P$-modules. Principal cofree modules and cofree modules are introduced as building blocks for working with $P$-modules.

\begin{defn}
A \define{poset} $P$ consists of a set $P$ equipped with a binary relation $\leq$ that is reflexive, antisymmetric, and transitive. 
A poset can be regarded as a category, where the objects are the elements of $P$ and where there is a morphism $a \to b$ whenever $a \leq b$.
\end{defn}

\begin{center}
\textbf{Henceforth, we assume all posets are finite.}
\end{center}

\begin{defn}
A \define{$P$-module} valued in a category $\cat$ is a functor $M: P \to \cat$. Let $\cat^P$ denote the category of $P$-modules, with morphisms given by natural transformations. Denote by $\Nat(M, N)$ the set of morphisms from $M$ to $N$.
\end{defn}

\begin{defn}
For $a \in P$ and $X \in \cat$, define the $P$-module $X^{\downarrow a}$ by
    \[
    X^{\downarrow a}(b) := \begin{cases} 
        X & \text{if } b \leq a, \\
        0 & \text{otherwise}.
    \end{cases}
    \]
The map $X^{\downarrow a}(b \leq c)$ is the identity on $X$ if $b \leq c$, and $0$ otherwise. Such $P$-modules are called \define{principal cofree modules}.
\end{defn}

\begin{defn}
A $P$-module $M$ is \define{cofree} if it is isomorphic to a product of principal cofree modules, i.e., if $M \cong \prod_{a \in P} X_a^{\downarrow a}$ for objects $X_a \in \cat$.
\end{defn}

Having introduced principal cofree modules, we now shift our focus to enriched hom-objects and tensor product constructions for $P$-modules.

\subsection{Enriched Hom}
 \label{sec:hom_objects}
This subsection defines enriched hom-objects and describes how the symmetric monoidal structure extends from the base category to the category of $P$-modules. We introduce tensor products of $P$-modules and establish an adjunction between the enriched hom functor and the tensor product functor.

If $\cat$ is an abelian category, then the category of $P$-modules, $\cat^P$, is also abelian. Moreover, if $(\cat, \otimes, \One)$ is a symmetric monoidal category, then $(\cat^P, \otimes, \One)$ inherits a symmetric monoidal structure as follows: for two $P$-modules $M$ and $N$, their tensor product $M \otimes N$ is given by the composition
\[
\begin{tikzcd}[column sep = 2em]
P \ar[rr, "M \times N"] && \cat \times \cat \ar[rr, "\otimes"] && \cat.
\end{tikzcd}
\]
The unit $P$-module, denoted $\One: P \to \cat$, assigns the object $\One$ to each element of $P$, with identity morphisms acting between them. The coherence conditions for this structure follow directly from those in $\cat$.

While we do not require $\cat^P$ to be closed, we do require it to be enriched over $\cat$. Specifically, for each $P$-module $M$, there must be a functor
\[
\uHom(M, -) : \cat^P \to \cat
\]
that is right adjoint to a functor of the form $- \otimes A : \cat \to \cat^P$, for each object $A \in \cat$.
\begin{defn}\label{def:enriched_hom}
Define the \define{enriched hom-object} between two $P$-modules $M$ and~$N$ as
\[
\uHom(M, N) := \int_{a \in P} \IHom_\cat(M(a), N(a)).
\]
This assignment extends, by the universal property of the end, to a functor
\[
\uHom(-, -) : (\cat^P)^\op \times \cat^P \to \cat.
\]
\end{defn}

Before proving the enriched hom adjunction, we introduce constant $P$-modules.
For an object $A \in \cat$, denote by $A^P : P \to \cat$ as the \emph{$A$-constant} $P$-module.
That is $A^P(a) = A$, for all $a \in P$, all the morphisms are set to the identity morphism on~$A$.

\begin{defn}
Define the \define{tensor product} $- \otimes M : \cat \to \cat^P$ as the functor induced by sending every 
object $A \in \cat$ to the tensor product of $P$-modules $A^P \otimes M$.
\end{defn}

\begin{prop}
\label{prop:enriched_hom}
The tensor product $-\otimes M : \cat \to \cat^P$ is left-adjoint to the 
enriched-hom $\uHom(M,-) : \cat^P \to \cat$. That is, for every object $A \in \cat$ and $P$-modules $M$ and $N$, there is a natural isomorphism of sets
\[
\Nat \big( A^P \otimes M, N \big) \cong \Hom_\cat \big( A, \uHom(M,N) \big).
\]
\end{prop}
\begin{proof}
For each $M,N\in \cat^P$ and $A\in \cat$, we have
		\begin{align*} 
			\Nat \big( A^P \otimes M ,N \big) & = \int_{a\in P} \Hom_\cat\big(A\otimes M(a),N(a)\big) && 
			\text{by definition of $A^P \otimes M$ and by \eqref{eq:end_formula_natural_transformations}} \\
			& \cong \int_{a\in P}\Hom_\cat\big(A,\IHom_\cat(M(a),N(a) \big) && \text{by the hom-tensor adjunction in $\cat$}\\
			& \cong \Hom_\cat \left(A,\int_{a\in P} \IHom_\cat \big( M(a),N(a) \big) \right)&& 
			\text{since $\Hom$ commutes with limits} \\
			& \cong\Hom_\cat \big( A,\uHom (M,N) \big).
		\end{align*}
These isomorphisms are natural in $A$, $M$, and $N$ which completes the proof.
\end{proof}

\begin{prop}
\label{prop:cofree_hom}
If $N \cong \prod_{a \in P} X_a^{\downarrow a}$ is a cofree $P$-module and $M$ is any $P$-module, then
\[
\uHom(M, N) \cong \prod_{a \in P} \IHom_\cat \big( M(a), X_a \big).
\]
\end{prop}

\begin{proof}
Since the functor $\uHom(M, -)$ is a right adjoint, it preserves limits, and in particular, it commutes with products. Therefore, it suffices to demonstrate that for every object $X \in \mathcal{C}$ and element $a \in P$, there is an isomorphism
\[
\uHom(M, X^{\downarrow a}) \cong \IHom_\mathcal{C}(M(a), X).
\]

Consider the case where $N = X^{\downarrow a}$. For each pair $b \leq c$ in $P$, examine the following commutative diagram:
\[
\begin{tikzcd}[column sep=1.5em]
\IHom_\mathcal{C}(M(a), X) \ar[r, "\phi_c"] \ar[d, "\phi_b"'] & \IHom_\mathcal{C}(M(c), X^{\downarrow a}(c)) \ar[d] \\
\IHom_\mathcal{C}(M(b), X^{\downarrow a}(b)) \ar[r] & \IHom_\mathcal{C}(M(b), X^{\downarrow a}(c))
\end{tikzcd}
\]
Here, for each $p \in P$, the morphism $\phi_p : \IHom_\mathcal{C}(M(a), X) \to \IHom_\mathcal{C}(M(p), X^{\downarrow a}(p))$ is defined by
\[
\phi_p = 
\begin{cases}
\IHom_\mathcal{C}(M(a \leq p), X) & \text{if } p \leq a, \\
0 & \text{otherwise}.
\end{cases}
\]
It is straightforward to verify that this diagram commutes for every $b \leq c$ in $P$. Consequently, the collection $\left( \IHom_\mathcal{C}(M(a), X), \phi_p \right)_{p \in P}$ forms a wedge for the functor $\IHom_\mathcal{C}(M(-), X^{\downarrow a}(-))$.

Now, let $L$ be any other such wedge with morphisms $e_p : L \to \IHom_\mathcal{C}(M(p), X^{\downarrow a}(p))$ for each $p \in P$. In particular, there is a morphism $e_a : L \to \IHom_\mathcal{C}(M(a), X)$ corresponding to $p = a$. It can be checked that for every $p \in P$, the following equality holds:
$\phi_p \circ e_a = e_p.$
This demonstrates that the morphisms $e_p$ uniquely factor through $\IHom_\mathcal{C}(M(a), X)$. Therefore, $\IHom_\mathcal{C}(M(a), X)$ is the universal wedge. 
\end{proof}

\subsection{Injective Resolutions}

We define principal injective modules and elementary injective modules, which are necessary for constructing injective resolutions for $P$-modules. This section parallels the theory of injective resolutions in the underlying abelian category.

\begin{defn}
For an injective object $Q \in \cat$ and $c \in P$, define the $P$-module~$Q^{\downarrow c}$ by:
\[
Q^{\downarrow c}(b) :=
\begin{cases}
Q & \text{if } b \leq c, \\
0 & \text{otherwise}.
\end{cases}
\]
The map $Q^{\downarrow c}(a \leq b)$ is the identity on $Q$ for $b \leq c$, and $0$ otherwise. Such $P$-modules are called \define{principal injective} $P$-modules.
\end{defn}

\begin{defn}
A $P$-module $M$ is called an \define{elementary injective} $P$-module if it is isomorphic to a product of principal injective $P$-modules. That is, $M \cong \prod_{a \in P} Q_a^{\downarrow a}$, where each $Q_a$ is an injective object in $\cat$.
\end{defn}

Clearly, principal injective modules are principal cofree modules, and elementary injective modules are cofree.
With the concept of elementary injective modules in hand, we now establish that every elementary injective module is injective in the category of $P$-modules.

\begin{prop}
\label{prop:elementary_injective}
Every elementary injective $P$-module is injective in $\cat^P$.
\end{prop}

\begin{proof}
Since products of injective objects are injective, it suffices to prove the proposition for a principal injective module 
$Q^{\downarrow b}$. 
Let $M$ and $N$ be $P$-modules, $\phi : M \hookrightarrow N$ a monomorphism, and $\psi : M \to Q^{\downarrow b}$ any morphism.
We construct a map of $P$-modules $\mu : N \to Q^{\downarrow b}$ such that $\phi \circ \mu = \psi$. 
Since $Q^{\downarrow b}(b) = Q$ is injective, there exists a map $\tilde{\mu}_b:N(b)\to Q$ in $\cat$ such that the diagram
\[
\begin{tikzcd}
M(b) \ar[r, hookrightarrow, "\phi_b"] \ar[d, "\psi_b"] & N(b) \ar[ld, dashrightarrow, "\tilde{\mu}_b"]\\
Q && 
\end{tikzcd}
\]
commutes. Define $\mu: N \to Q^{\downarrow b}$ by setting $\mu_b = \tilde{\mu}_b$ and then extending to all of $P$ as follows:
\[
\mu_a := \begin{cases} 
\mu_b \circ N(a \leq b) & \textup{if } a \leq b, \\
0 & \textup{otherwise}.
\end{cases}
\]
It is straightforward to verify that $\mu$ is a natural transformation and that $\phi \circ \mu = \psi$.
\end{proof}

\subsection{Enough Injectives}

We show that if the base category has enough injectives, then the category of $P$-modules inherits this property, enabling us to construct injective resolutions for all $P$-modules.

\begin{prop}
\label{prop:enough_injectives_modules}
If $\cat$ has enough injectives, then the category $\cat^P$ has enough injectives.
\end{prop}

\begin{proof}
We show that $\cat^P$ has enough injectives by constructing, for a fixed $P$-module $M$, an injective $P$-module $Q$ 
and a monomorphism $j : M \mono Q$.

Since $\cat$ has enough injectives, for each $a \in P$, there exists an injective object $Q_a \in \cat$ and a monomorphism 
$i_a: M(a) \mono Q_a$. Set 
\[
Q := \prod_{a \in P} Q_a^{\downarrow a}.
\]
By Proposition \ref{prop:elementary_injective}, $Q$ is an elementary injective $P$-module and hence is injective.

For each $b \in P$, define a map $j_b : M \to Q^{\downarrow b}$ by
\[
j_b(a) = 
\begin{cases}
i_b \circ M(a \leq b) & \text{if } a \leq b, \\
0 & \text{otherwise}.
\end{cases}
\]
It is straightforward to verify that each $j_b$ is a natural transformation. Note that $j_b(b)$ is a monomorphism. 
Finally, the desired map $j : M \mono Q$ is obtained using the universal property of the product.
\end{proof}

By Proposition \ref{prop:enough_injectives_modules}, every $P$-module admits a monomorphism into an elementary injective module, leading to the following corollary.

\begin{cor}
If $\cat$ is an abelian category with enough injectives, then every $P$-module $M \in \cat^P$ has an injective resolution by elementary injective $P$-modules.
\end{cor}

\section{Derived Hom Functors}
\label{sec:derived_hom}
In this section, we introduce the Ext-functor as a derived functor of the enriched hom. 
Next, we introduce cofree resolutions as a computational tool. 
Finally, we explore the Euler characteristic of Ext-functors.

\subsection{Ext Functors and Injective Resolutions}
This subsection introduces Ext-functors, which arise as derived functors of the enriched hom.

Recall Proposition~\ref{prop:enriched_hom}: for a fixed $P$-module $M$, the tensor product functor 
\[
- \otimes M : \cat \to \cat^P 
\]
is left adjoint to the enriched hom functor 
\[
\uHom(M, -) : \cat^P \to \cat.
\] 
This adjunction makes $\uHom(M, -)$ a left-exact functor between the two abelian categories.

Now, let $N$ be a second $P$-module. Consider an injective resolution
\[
0 \to N \hookrightarrow I^0 \to I^1 \to I^2 \to \cdots,
\]
and apply the functor $\uHom(M, -)$ to obtain the following cochain complex in $\cat$:
\[
0\to \uHom(M, I^0) \to \uHom(M, I^1) \to \uHom(M, I^2) \to \cdots.
\]
The cohomology objects of this complex, denoted by $R^d \uHom(M, N)$, are the higher derived functors of the enriched hom and are commonly referred to as \emph{Ext-functors}. These are denoted by
\[
\uExt^d(M, N).
\]

We are particularly interested in the Ext-functor for a specific class of indicator $P$-modules $M$, which we now describe.

\begin{defn}
A subposet $Z \subset P$ is called a \define{spread} if, for all $a \leq c$ in $Z$ and for all $a \leq b \leq c$ in $P$, we have $b \in Z$. Define the $P$-module $\One_Z : P \to \cat$ as follows:
\[
\One_Z(b) :=
\begin{cases}
\One & \text{if } b \in Z, \\
0 & \text{otherwise}.
\end{cases}
\]
Additionally, the morphism $\One_Z(a \leq b)$ is the identity morphism for $a \leq b$ in $Z$ and zero otherwise.
\end{defn}

As a direct consequence of Propositions~\ref{prop:hom_one} and~\ref{prop:cofree_hom}, we have the following corollary.

\begin{cor}
\label{cor:spread_hom}
For a spread $Z \subset P$ and a cofree $P$-module $N \cong \prod_{a \in P} X_a^{\downarrow a}$, we have
\[
\uHom (\One_Z, N) \cong \prod_{a \in Z} X_a.
\]
\end{cor}

\begin{center}
\textbf{Henceforth, all Ext-functors will be of the form $\uExt^\ast (\One_Z, N)$, where $Z \subset P$ is a spread.}
\end{center}

\subsection{Standard Cofree Resolutions and Ext Calculations}
This subsection focuses on calculating Ext-functors using cofree resolutions. 
While injective resolutions are theoretically useful for defining the Ext-functor, cofree resolutions offer computational advantages, particularly for explicit calculations. 

A \emph{cofree resolution} of a $P$-module $N$ is an exact sequence in $\cat^P$:
\[
0 \to N \to F^0 \to F^1 \to F^2 \to \cdots
\]
where each $F^d$ is a cofree $P$-module. A specific cofree resolution, called the \emph{standard cofree resolution}, is particularly useful for computations.

\begin{defn}
The \define{order complex} of a poset $P$, denoted by $\Delta P$, is a simplicial complex whose $n$-simplices are chains $\sigma = a_0 < \cdots < a_n$ consisting of $n+1$ distinct elements of $P$. A simplex $\sigma$ is a \define{face} of a simplex $\tau$, written $\sigma \unlhd \tau$, if $\sigma$ is a subchain of $\tau$.
\end{defn}

We use $\sigma \lhd_i \tau$ to indicate that $\dim \tau - \dim \sigma = i$. For a simplex $\tau = a_0 < \cdots < a_d$, we denote $\min \tau$ as $a_0$ and $\max \tau$ as $a_d$. If $\sigma \unlhd \tau$, then $\min \tau \leq \min \sigma$ and $\max \sigma \leq \max \tau$. 

\begin{defn}
The \define{standard cofree resolution} of a $P$-module $N$ is the exact sequence:
\begin{equation}
\label{eq:cofre_res}
\begin{tikzcd}[column sep = 2em]
0 \ar[r] & N \ar[r, hook,"\epsilon"] & F^0 N \ar[r, "\delta^0"] & F^1 N \ar[r, "\delta^1"] & \cdots
\end{tikzcd}
\end{equation}
Define $F^d N : P \to \cat$ as
\[
F^d N := \prod_{\sigma \in \Delta P : \dim \sigma = d} N(\max \sigma)^{\downarrow \min \sigma}.
\]
For $\sigma \lhd_1 \tau$, define the map
\[
\delta^d |_{\sigma \lhd_1 \tau}: N(\max \sigma)^{\downarrow \min \sigma} \to N(\max \tau)^{\downarrow \min \tau}
\]
as
\[
\delta^d |_{\sigma \lhd_1 \tau}(a) = 
\begin{cases}
[\tau : \sigma] \cdot N(\max \sigma \leq \max \tau) & \textup{if } a \leq \min \sigma, \\
0 & \textup{otherwise}.
\end{cases}
\]
The universal property of the product induces the full coboundary morphism $\delta^d : F^d N \to F^{d+1} N$.

The canonical monomorphism $\epsilon: N \mono F^0 N$ is defined as follows. For a $0$-simplex $\sigma \in \Delta P$ (i.e., an element $b = \max \sigma = \min \sigma$), define the component
\[
\epsilon |_\sigma : N \to N(\max \sigma)^{\downarrow \min \sigma}
\]
as
\[
\epsilon |_\sigma (a) = 
\begin{cases}
N(a \leq \max \sigma) & \textup{if } a \leq \min \sigma, \\
0 & \textup{otherwise}.
\end{cases}
\]
The universal property of the product induces the full monomorphism~$\epsilon$.
\end{defn}

\begin{prop}[\cite{Deheuvels1962}]
The standard cofree resolution of a $P$-module is exact.
\end{prop}

%

Now consider an elementary cofree $P$-module $X_a^{\downarrow a}$. Since $\cat$ has enough injectives, $X_a$ admits an injective resolution:
\[
0 \to X_a \hookrightarrow I^0_a \to I^1_a \to I^2_a \to \cdots,
\]
which induces an injective resolution of the elementary $P$-module:
\[
0 \to X_a^{\downarrow a} \hookrightarrow (I^0_a)^{\downarrow a} \to (I^1_a)^{\downarrow a} \to (I^2_a)^{\downarrow a} \to \cdots.
\]

Applying the functor $\uHom(\One_Z, -)$, and using Propositions~\ref{prop:cofree_hom} and~\ref{prop:hom_one}, we obtain, assuming $a \in Z$, the following acyclic cochain complex:
\begin{equation}
\label{eq:mobius_cofree}
0\to I^0_a \to I^1_a \to I^2_a \to \cdots
\end{equation}
If $a \notin Z$, the resulting chain complex is zero.
Thus, $\uExt^d(\One_Z, X_a^{\downarrow a}) = 0$ for all $d > 0$. In other words, elementary cofree $P$-modules, and hence cofree $P$-modules in general, are $\uHom(\One_Z, -)$-acyclic, leading to the following proposition.

\begin{prop}{\cite[Theorem III.6.16]{GelfandManin2003}}
\label{prop:co_free}
Let $0 \to N \to F^0 N \to F^1 N \to F^2 N \to \cdots$ be the standard cofree resolution of a $P$-module $N$, and let
$0 \to N \to I^0 \to I^1 \to I^2 \to \cdots$
be any injective resolution of $N$. 
For all spreads $Z \subset P$, the two cochain complexes
\[
0 \to\uHom(\One_Z, I^0) \to \uHom(\One_Z, I^1) \to \uHom(\One_Z, I^2) \to \cdots
\]
and
\[
0\to \uHom(\One_Z, F^0 N) \to \uHom(\One_Z, F^1 N) \to \uHom(\One_Z, F^2 N) \to \cdots
\]
are quasi-isomorphic, i.e., their cohomology objects are isomorphic.
\end{prop}

\subsection{Euler Characteristic of Ext-Functors}
Having established how Ext-functors are computed using cofree resolutions, we now define the Euler characteristic as an element of the Grothendieck ring. 

\begin{defn}
    The \define{Euler characteristic} of the Ext-functor $\uExt^\ast(\One_Z, N)$ is defined as:
    \[
    \chi(\One_Z, N) := \sum_{d > 0} (-1)^d \big[ \uExt^d(\One_Z, N) \big].
    \]
\end{defn}

The following proposition gives an explicit formula for computing $\chi(\One_Z, N)$.

\begin{prop}
\label{prop:euler_indicator}
For a spread $Z \subset P$ and a $P$-module $N$, the Euler characteristic is given by:
\[
\chi(\One_Z, N) = \sum_{d \geq 0} (-1)^d \sum_{\substack{\sigma \in \Delta P \\ \dim \sigma = d \\ \min \sigma \in Z}} \left[ N(\max \sigma) \right].
\]
\end{prop}

\begin{proof}
We compute the Euler characteristic by analyzing the cochain complex:
\begin{equation}\label{eq:standard_cofree_hom}
0 \stackrel{\delta^{-1}}{\to} \uHom(\One_Z, F^0 N) \stackrel{\delta^0}{\to} \uHom(\One_Z, F^1 N) \stackrel{\delta^1}{\to} \uHom(\One_Z, F^2 N) \stackrel{\delta^2}{\to} \cdots
\end{equation}
For each degree \(d\), let 
\[
Z_d(\One_Z, N) := \ker(\delta^d) \quad \text{and} \quad B_d(\One_Z, N) := \im(\delta^{d-1}),
\]
yielding the short exact sequence:
\[
0 \to Z_d(\One_Z, N) \hookrightarrow \uHom(\One_Z, F^d N) \twoheadrightarrow B_{d+1}(\One_Z, N) \to 0.
\]
By Corollary~\ref{cor:spread_hom}, the middle term can be expressed as:
\[
\uHom(\One_Z, F^d N) \cong \prod_{\substack{\sigma \in \Delta P \\ \dim \sigma = d \\ \min \sigma \in Z}} N(\max \sigma).
\]
Thus, in the Grothendieck ring:
\[
\sum_{\substack{\sigma \in \Delta P \\ \dim \sigma = d \\ \min \sigma \in Z}} \left[ N(\max \sigma) \right] = \left[ \uHom(\One_Z, F^d N) \right] = [Z_d(\One_Z, N)] + [B_{d+1}(\One_Z, N)].
\]

Since the cohomology of Equation~\eqref{eq:standard_cofree_hom} computes the Ext-functors $\uExt^d(\One_Z, N)$, we have the following short exact sequence:
\[
0 \to B_d(\One_Z, N) \hookrightarrow Z_d(\One_Z, N) \twoheadrightarrow \uExt^d(\One_Z, N) \to 0,
\]
which gives:
\[
[\uExt^d(\One_Z, N)] = [Z_d(\One_Z, N)] - [B_d(\One_Z, N)].
\]

Using the fact that \( [B_0(\One_Z, N)] = 0 \), we obtain:
\begin{align*}
\sum_{\substack{\sigma \in \Delta P \\ \dim \sigma = d \\ \min \sigma \in Z}} (-1)^d \left[ N(\max \sigma) \right] 
&= \sum_{d \geq 0} (-1)^d \left( [Z_d(\One_Z, N)] + [B_{d+1}(\One_Z, N)] \right) \\
&= \sum_{d \geq 0} (-1)^d [Z_d(\One_Z, N)] - \sum_{d \geq 1} (-1)^d [B_d(\One_Z, N)] \\
&= \sum_{d \geq 0} (-1)^d \left( [Z_d(\One_Z, N)] - [B_d(\One_Z, N)] \right) \\
&= \sum_{d \geq 0} (-1)^d [\uExt^d(\One_Z, N)] \\
&= \chi(\One_Z, N). \qedhere
\end{align*}
\end{proof}

\section{M\"obius Cohomology}
\label{sec:mobius_cohomology}
In this section, we introduce Möbius cohomology as a special case of Ext-cohomology that applies to indicator modules associated with individual elements of a poset. 
The discussion concludes with an overview of Möbius inversions and how Möbius cohomology serves as its categorification.
\subsection{Definition}
We begin by specializing the notion of indicator modules $\One_Z$ to singleton sets. For any element \( a \in P \), the set \( \{a\} \subset P \) is a spread, allowing us to define its indicator module in the same way as for more general indicator $P$-modules. For simplicity, we denote the $P$-module $\One_{\{a\}}$ by $\One_a$.

\begin{defn}
The \define{M\"obius cohomology} of a $P$-module $N$ at $a \in P$ is given by the Ext-functor:
\[
\uExt^d(\One_a, N).
\]
\end{defn}

By Equation~\ref{eq:standard_cofree_hom} and Corollary~\ref{cor:spread_hom}, the M\"obius cohomology corresponds to the cohomology of the following cochain complex:
\begin{equation}\label{eq:standard_mobius_complex}
0 \to \prod_{\substack{\sigma \in \Delta P \\ \dim \sigma = 0 \\ \min \sigma = a}} N(\max \sigma)^{\downarrow a} 
\stackrel{\delta^0}{\to} \prod_{\substack{\sigma \in \Delta P \\ \dim \sigma = 1 \\ \min \sigma = a}} N(\max \sigma)^{\downarrow a} 
\stackrel{\delta^1}{\to} \prod_{\substack{\sigma \in \Delta P \\ \dim \sigma = 2 \\ \min \sigma = a}} N(\max \sigma)^{\downarrow a} 
\stackrel{\delta^2}{\to} \cdots
\end{equation}

\begin{rmk}
M\"obius homology, introduced by Patel and Skraba~\cite{patel2023mobius_homology}, provides a dual view. It is defined as the homology of the chain complex:
\[
\cdots \longrightarrow \prod_{\substack{\sigma \in \Delta P \\ \dim \sigma = 2 \\ \max \sigma = a}} N(\min \sigma)^{\downarrow a} 
\longrightarrow \prod_{\substack{\sigma \in \Delta P \\ \dim \sigma = 1 \\ \max \sigma = a}} N(\min \sigma)^{\downarrow a} 
\longrightarrow \prod_{\substack{\sigma \in \Delta P \\ \dim \sigma = 0 \\ \max \sigma = a}} N(\min \sigma)^{\downarrow a} 
\longrightarrow 0.
\]
The above chain complex is precisely the dual of Equation~\ref{eq:standard_mobius_complex}, demonstrating that M\"obius cohomology and homology are dual constructions.
\end{rmk}

\subsection{Background: M\"obius Inversion}
This subsection introduces the classical combinatorial tool of Möbius inversion, which, in the next subsection, is connected to Möbius cohomology through the Euler characteristic.

For any pair \(a \leq c\) in a poset $P$, the \emph{interval} between them is defined as:
\[
[a, c] := \{b \in P \mid a \leq b \leq c\}.
\]
Let $\Int P$ denote the set of all intervals in $P$. 

The \emph{$\Z$-incidence algebra} of $P$, denoted $\Inc(P)$, consists of functions $\alpha: \Int P \to \Z$ with operations of scaling, addition, and multiplication. Multiplication is given by:
\[
(\alpha \ast \beta)[a,c] := \sum_{b : a \leq b \leq c} \alpha[a,b] \cdot \beta[b,c].
\]
The multiplicative identity in this algebra is:
\[
\One[a, b] = 
\begin{cases}
1 & \text{if } a = b, \\
0 & \text{otherwise}.
\end{cases}
\]
The \emph{zeta function}, denoted by $\zeta$, assigns the value $1$ to every interval: $\zeta[a, b] = 1$. The inverse of $\zeta$, denoted $\mu$, is known as the \emph{M\"obius function}. 

\begin{lem}[Philip Hall’s Theorem, Prop 3.8.5 \cite{stanley_2011}]
\label{lem:Hall}
For \(a \leq b\) in $P$, let $n_i$ denote the number of chains of length $i$ from $a$ to $b$. Then:
\[
\mu[a, b] = \sum_{i > 0} (-1)^i \cdot n_i.
\]
\end{lem}

\begin{defn}
Let $P$ be a finite poset, and let $\Gr$ be a unital commutative ring. The \define{(upper) M\"obius inversion} of a function $f : P \to \Gr$ is defined as:
\begin{equation}
\label{eq:mobius_inversion}
\partial f(a) := \sum_{b : a \leq b} f(b) \cdot \mu[a, b].
\end{equation}
\end{defn}

The M\"obius inversion is the unique $\Gr$-valued function on $P$ satisfying the following identity for all \(a \in P\):
\begin{align*}
\sum_{b : a \leq b} \partial f(b) &= \sum_{b : a \leq b} \partial f(b) \cdot \zeta[a,b] \\
&= \sum_{b : a \leq b} \left( \sum_{c : b \leq c} f(c) \cdot \mu[b,c] \right) \cdot \zeta[a,b] \\
&= \sum_{c : a \leq c} f(a) \left( \sum_{b : a \leq b \leq c} \zeta[a,b] \cdot \mu[b,c] \right) \\
&= \sum_{c : a \leq c} f(c) \cdot \One[a,c] \\
&= f(a).
\end{align*}

Note that the set $\Gr^P$ of all functions $P\to \Gr$ forms an $\Gr$-module. M\"obius inversion can then be viewed as a module homomorphism $\partial:\Gr^P\to \Gr^P$. We will also write $\partial_P$ when we wish to emphasize the poset over which the M\"obius inversion is being taken.

\begin{rmk}
While our focus is on the upper Möbius inversion, which sums over elements greater than or equal to a given element \(a \in P\), there is an alternative form known as the \emph{lower Möbius inversion}. 
Specifically, for a function \(f : P \to \Gr\), the lower inversion is defined as:
\[
\partial_{-} f(a) = \sum_{b : b \leq a} f(b) \cdot \mu[b, a].
\]
Although this form is not required in our framework, it plays a dual role to the upper inversion.
\end{rmk}

\subsection{Euler Characteristic of M\"obius Cohomology}
We now examine the relationship between Möbius cohomology and Möbius inversion through the Euler characteristic.

\begin{defn}
The \define{dimension function} of a $P$-module $N$ is the function \(n : P \to \Gr(\cat)\) given by:
\[
n(b) := [N(b)].
\]
\end{defn}

\begin{thm}
\label{thm:cohom_euler_char}
For any $P$-module $N$ and \(a \in P\), the M\"obius inversion of the dimension function equals the Euler characteristic of the M\"obius cohomology:
\[
\partial n(a) = \sum_{d \geq 0} (-1)^d \big[\uExt^d(\One_a, N) \big].
\]
\end{thm}

\begin{proof} For each $a\in P$, we have
\begin{align*}
\sum_{d \geq 0} (-1)^d \big[ \uExt (\One_a, N) \big ] &= 
	\prod_{\substack{\sigma \in \Delta P \\ \dim \sigma = d \\ \min \sigma = a}} N(\max \sigma) 
&&  \text{by Proposition~\ref{prop:euler_indicator}} \\
	&= \sum_{d \geq 0} (-1)^d \sum_{b : a \leq b} n_d(a,b) \cdot [N(b)] \\
	&= \sum_{b : a \leq b} \sum_{d \geq 0} (-1)^d n_d(a,b) \cdot [N(b)] \\
	&= \sum_{b : a \leq b} \mu(a,b) \cdot [N(b)] && \text{by Lemma~\ref{lem:Hall}} \\
	&= \partial m(a). && \text{by Equation~(\ref{eq:mobius_inversion})} \qedhere
\end{align*}
\end{proof}

\begin{rmk}
The lower Möbius inversion arises naturally in the decategorification of Möbius homology. 
This reflects the duality between homology and cohomology: while Möbius cohomology decategorifies to the upper inversion, Möbius homology decategorifies to the lower inversion.
\end{rmk}

\section{Galois Connections}
\label{sec:Galois_connections}

This section examines Galois connections, which are adjunctions between posets. We begin by defining the fundamental functors between poset modules that arise from monotone functions, followed by a discussion of the adjoint functors induced by Galois connections. 
The central focus of this section is Theorem~\ref{thm:categorical_RGCT}, which serves as a categorification of Rota’s classical theorem on Galois connections and Möbius functions. 
We then review Rota’s classical theorem before demonstrating how it is categorified by our theorem.

\subsection{Adjoint Functors from Monotone Functions}
In this subsection, we explore the functors that naturally arise from monotone functions between posets.
These functors—pushforward, pushforward with open supports, and pullback—allow us to transfer structure between categories of modules over different posets.

Let \( f: P \to Q \) be a monotone function between two posets. A function \( f \) is monotone if \( a \leq b \) in \( P \) implies \( f(a) \leq f(b) \) in \( Q \). 
A monotone function generates three distinct but related functors:

\begin{itemize}
    \item The \define{pushforward} functor $f_\ast : \cat^P \to \cat^Q$ sends a $P$-module $M$ to the $Q$-module $f_\ast M$, where
    \[
    f_\ast M(x) := \lim_{a \in P: f(a) \geq x} M(a).
    \]

    \item The \define{pushforward with open supports} functor $f_\dagger : \cat^P \to \cat^Q$ sends a $P$-module $M$ to the $Q$-module $f_\dagger M$, where
    \[
    f_\dagger M(x) := \colim_{a \in P: f(a) \leq x} M(a).
    \]

    \item The \define{pullback} functor $f^\ast : \cat^Q \to \cat^P$ sends a $Q$-module $N$ to the $P$-module $f^\ast N$, where
    \[
    f^\ast N := N \circ f.
    \]
\end{itemize}

The relationships among these functors are summarized in the following diagram:
\[
\begin{tikzcd}
    \cat^P \ar[rr, bend left, "f_\ast"] \ar[rr, bend right, "f_\dagger"'] && \cat^Q \ar[ll, "f^\ast"']
\end{tikzcd}
\]

\begin{prop}\label{prop:P-module_functors_adjoint}
\cite[Theorems 3.14 and 3.15]{curry2018} For every monotone function \( f : P \to Q \), the following adjunctions hold:
\[
f_\dagger \dashv f^\ast \quad \text{and} \quad f^\ast \dashv f_\ast.
\]
\end{prop}

Since right adjoints are left-exact, and left adjoints are right-exact, the following properties hold:
\begin{itemize}
    \item \( f_\ast \) is left-exact.
    \item \( f^\ast \) is exact (both left-exact and right-exact).
    \item \( f_\dagger \) is right-exact.
\end{itemize}

\subsection{Galois Connections and Functorial Equality}
In this subsection, we explore how Galois connections induce equalities of functors.

\begin{defn}
A \define{Galois connection} between two posets \( P \) and \( Q \) consists of two monotone functions \( f: P \to Q \) and \( g: Q \to P \) such that:
\[
f(a) \leq x \iff a \leq g(x),
\]
for all \( a \in P \) and \( x \in Q \). We write \( f : P \rightleftarrows Q : g \) and refer to \( f \) as the \emph{left adjoint} and \( g \) as the \emph{right adjoint}.
\end{defn}

\begin{prop}\label{prop:galois_connection_functor_equalites}
For a Galois connection \( f : P \rightleftarrows Q : g \), the following functorial equalities hold:
\[
f^\ast = g_\ast \quad \text{and} \quad f_\dagger = g^\ast.
\]
\end{prop}

\begin{proof}
For \( N \in \cat^Q \) and \( a \in P \), we have:
\[
g_\ast N(a) = \lim_{x \in Q : a \leq g(x)} N(x) = N(f(a)) = f^\ast N(a).
\]
Similarly, for \( M \in \cat^P \) and \( x \in Q \), we obtain:
\[
f_\dagger M(x) = \colim_{a \in P : f(a) \leq x} M(a) = M(g(x)) = g^\ast M(x). \qedhere
\]
\end{proof}

\subsection{Categorical Galois Connection Theorem}
Galois connections naturally give rise to an enriched adjunction leading us to the main theorem of this section.

Given a Galois connection \( f: P \rightleftarrows Q : g \), Propositions~\ref{prop:P-module_functors_adjoint} and \ref{prop:galois_connection_functor_equalites} establish the following adjunctions:
\[
     g^*  = f_\dagger\dashv f^* \quad \text{and} \quad  g_* = f^* \dashv f_*.
\]
Thus, we have the following natural isomorphisms:
\[
    \Nat(g^* M, N) \cong \Nat(M, f^* N) \quad \text{and} \quad \Nat(g_\ast N, M) \cong \Nat(N, f_\ast M),
\]
for all \( M \in \cat^P \) and \( N \in \cat^Q \). 

In what follows, we extend these adjunctions to the enriched hom setting, formalized in Theorem~\ref{thm:categorical_RGCT}. To do so, we first establish the necessary and sufficient conditions under which an adjunction between functors extends to the enriched hom.

\begin{lem}
\label{lem:tensor_hom}
Let $L : \cat^P \rightleftarrows \cat^Q : R$ be a pair of adjoint functors with $L \dashv R$. Then
\[
\uHom(L(M), N) \cong \uHom(M, R(N))
\]
for all $M \in \cat^P$ and $N \in \cat^Q$ if and only if 
\[
L(A^P \otimes M) \cong A^Q \otimes L(M)
\]
for all $M \in \cat^P$. Moreover, the first isomorphism is natural in $M$ and $N$ if and only if the second is natural in $M$.
\end{lem}

\begin{proof}
Assume that $\uHom(L(M), N) \cong \uHom(M, R(N))$ for all $M \in \cat^P$ and $N \in \cat^Q$. Then for any $M \in \cat^P$ and $N \in \cat^Q$, we have:
\begin{align*}
    \Nat(L(A^P \otimes M), N) &\cong \Nat(A^P \otimes M, R(N)) && \text{since } L \dashv R \\
    &\cong \Hom_\cat(A, \uHom(M, R(N))) && \text{by Proposition~\ref{prop:enriched_hom}} \\
    &\cong \Hom_\cat(A, \uHom(L(M), N)) && \text{by assumption} \\
    &\cong \Nat(A^Q \otimes L(M), N) && \text{by Proposition~\ref{prop:enriched_hom}}.
\end{align*}
Since $N \in \cat^Q$ was arbitrary, by the Yoneda lemma, we have $L(A^P \otimes M) \cong A^Q \otimes L(M)$. These isomorphisms are natural in $M$ and $N$, assuming the isomorphism $\uHom(L(M), N) \cong \uHom(M, R(N))$ is natural. Hence, the isomorphism $L(A^P \otimes M) \cong A^Q \otimes L(M)$ is natural in $M$.

Conversely, suppose $L(A^P \otimes M) \cong A^Q \otimes L(M)$ for all $M \in \cat^P$. For every object $A \in \cat$, we have:
\begin{align*}
    \Hom_\cat(A, \uHom(L(M), N)) &\cong \Nat(A^Q \otimes L(M), N) && \text{by Proposition~\ref{prop:enriched_hom}} \\
    &\cong \Nat(L(A^P \otimes M), N) && \text{by assumption} \\
    &\cong \Nat(A^P \otimes M, R(N)) && \text{since } L \dashv R \\
    &\cong \Hom_\cat(A, \uHom(M, R(N))) && \text{by Proposition~\ref{prop:enriched_hom}}.
\end{align*}
By the Yoneda lemma, we conclude that $\uHom(L(M), N) \cong \uHom(M, R(N))$. All of these isomorphisms are natural in $M$ and $N$, provided that $L(A^P \otimes M) \cong A^Q \otimes L(M)$ is natural in $M$. This completes the proof.
\end{proof}

Next, we verify that pullbacks and pushforwards with open supports satisfy the conditions of the preceding lemma.
\begin{lem} 
\label{lem:daggerpushforward_and_pullback_commute_with_tensor}
Let $h : P \to Q$ be a monotone function between posets. Then $h_\dagger(A^P \otimes M) \cong A^Q \otimes h_\dagger M$ and $h^*(A^Q \otimes N) \cong A^P \otimes h^* N$ for all $M \in \cat^P$ and $N \in \cat^Q$. Moreover, these isomorphisms are natural in $M$ and~$N$. 
\end{lem}

\begin{proof}
By Proposition~\ref{prop:enriched_hom}, the tensor product $-\otimes M : \cat \to \cat^P$ is a left-adjoint and therefore
commute with colimits.
Let us verify the isomorphism $h_\dagger(A^P \otimes M) \cong A^Q \otimes h_\dagger M$. For each $x \in Q$, we have
\begin{align*}
    h_\dagger(A^P \otimes M)(x) &= \colim_{a \in P : h(a) \leq x} (A^P \otimes M)(a) && \textup{by definition of $h_\dagger$}\\
    &\cong A \otimes \colim_{a \in P : h(a) \leq x} M(a) && \textup{since $\otimes$ commutes with colimts}\\
    &= (A^Q \otimes h_\dagger M)(x) &&\textup{by definition of $h_\dagger$.}
\end{align*}
Each of these isomorphisms is natural in $x$ and $M$, so we obtain an isomorphism $h_\dagger(A^P \otimes M) \cong A^Q \otimes h_\dagger M$ of $Q$-modules, natural in $M$.

Next, we verify the isomorphism $h^*(A^Q \otimes N) \cong A^P \otimes h^* N$. For each $a \in P$, we have:
\[
h^*(A^Q \otimes N)(a) = (A^Q \otimes N)(h(a)) = A \otimes N(h(a)) = A \otimes h^* N(a) = (A^P \otimes h^* N)(a).
\]
Thus, $h^*(A^Q \otimes N)$ and $A^P \otimes h^* N$ are pointwise equal. Moreover, the internal maps of both $P$-modules are the same, so we in fact have the equality $h^*(A^Q \otimes N) = A^P \otimes h^* N$, which is trivially natural in~$N$.
\end{proof}

\begin{thm}[Categorical Rota's Galois Connection Theorem] \label{thm:categorical_RGCT}
Let \( f : P \rightleftarrows Q : g \) be a Galois connection. Then:
\[
\uHom(g^* M, N) \cong \uHom(M, f^* N) \quad \text{and} \quad \uHom(g_* N, M) \cong \uHom(N, f_* M),
\]
for all \( M \in \cat^P \) and \( N \in \cat^Q \).
\end{thm}

\begin{proof}
By Lemma~\ref{lem:tensor_hom} and Lemma~\ref{lem:daggerpushforward_and_pullback_commute_with_tensor}, we have:
\[
\uHom(h_\dagger M, N) \cong \uHom(M, h^* N) \quad \text{and} \quad \uHom(h^* N, M) \cong \uHom(N, h_* M).
\]
Using Proposition~\ref{prop:galois_connection_functor_equalites}, we substitute \( f_\dagger = g^\ast \) and \( f^\ast = g_\ast \):
\[
\uHom(g^\ast M, N) \cong \uHom(M, f^\ast N), \quad \text{and} \quad \uHom(g_\ast N, M) \cong \uHom(N, f_\ast M). \qedhere
\]
\end{proof}

\subsection{Background: Rota's Theorem}
In this subsection, we review Rota's original theorem, which establishes a fundamental relationship between Möbius functions on two posets connected by a Galois connection. 

\begin{thm}[Rota's Galois Connection Theorem~\cite{rota1964foundations}]
\label{thm:rota_classic}
Let \( f : P \rightleftarrows Q : g \) be a Galois connection. For all \( a \in P \) and \( y \in Q \), the following equality holds:
\[
\sum_{x \in Q : g(x) = a} \mu_Q(x, y) = \sum_{b \in P : f(b) = y} \mu_P(a, b).
\]
\end{thm}

This theorem connects the Möbius functions \( \mu_P \) and \( \mu_Q \) associated with the respective posets \( P \) and \( Q \). 
We now present an equivalent result using M\"obius inversions.

Given a monotone function \( f : P \to Q \), we define two key operations:

\begin{itemize}
    \item The \define{pushforward} \( f_\# m : Q \to \Gr \) of a function \( m : P \to \Gr \) is given by:
    \[
        f_\# m(z) := \sum_{a \in f^{-1}(z)} m(a).
    \]
    \item The \define{pullback} \( f^\# n : P \to \Gr \) of a function \( n : Q \to \Gr \) is defined by:
    \[
        f^\# n(a) := n(f(a)).
    \]
\end{itemize}
Note that the pushforward and pullback gives rise to $\Gr$-module homomorphisms $f_\sharp:\Gr^P\to \Gr^Q$ and $f^\sharp: \Gr^Q\to \Gr^P$, respectively.

We now introduce an equivalent result to Rota's theorem, which highlights how these pushforward and pullback operations interact with Möbius inversions. This alternative formulation was first observed by Aziz G\"ulen and Alex McCleary for the dual setting of lower Möbius inversions~\cite{GulenMcCleary}.

\begin{thm}
\label{thm:rota_second_ver}
For a Galois connection \( f : P \rightleftarrows Q : g \), the following identity holds:
\[
\partial_P \circ f^\# = g_\# \circ \partial_Q.
\]
\end{thm}

\begin{proof}
Let \( n : Q \to \Gr \) be an arbitrary function. For each \( a \in P \), we have

\begin{align*}
\sum_{b : a \leq b} ( g_\#  \circ \partial_Q)(n)(b) 
&= \sum_{b : a \leq b} \sum_{q \in g^{-1}(b)} \partial_Q n(q) \\
&= \sum_{q \in Q : a \leq g(q)} \partial_Q n(q) \\
&= \sum_{q \in Q : f(a) \leq q} \partial_Q n(q) && \text{since \( f \dashv g \)} \\
&= n(f(a)) \\
&= f^\sharp n(a).
\end{align*}
By uniqueness of the M\"obius inversion, it follows that $(g_\sharp \circ \partial_Q)n = (\partial_P \circ f^\sharp)n$. Since this is true for every $n:Q\to \Gr$, the result follows.
\end{proof}

The equivalence between Theorem~\ref{thm:rota_second_ver} and Theorem~\ref{thm:rota_classic} can be demonstrated by applying the Möbius inversion formula to a specific case. Choose \( a \in P \) and \( y \in Q \), and let \( 1_y : Q \to \Gr \) be the function defined by:
\[
1_y(z) =
\begin{cases}
1 & \text{if } z = y, \\
0 & \text{otherwise}.
\end{cases}
\]
Evaluating both sides of the identity in Theorem~\ref{thm:rota_second_ver} using this function, we obtain, for all $a\in P$ and $y\in Q$:

\begin{alignat*}{3}
&&(g_\# \circ \partial_Q)(1_y)(a) &= (\partial_P \circ f^\#)(1_y)(a) \\
\iff \qquad&&\sum_{x \in g^{-1}(a)} \partial_Q 1_y(x) 
&= \sum_{b \in P : a \leq b} (1_y \circ f)(b) \cdot \mu_P(a, b) \\
\iff \qquad&&\sum_{x \in g^{-1}(a)} \sum_{z \in Q : x \leq z} 1_y(z) \cdot \mu_Q(x, z) 
&= \sum_{b \in P : f(b) = y} \mu_P(a, b) \\
\iff \qquad&&\sum_{x \in g^{-1}(a)} \mu_Q(x, y) 
&= \sum_{b \in P : f(b) = y} \mu_P(a, b) \\
\iff \qquad&&\sum_{x \in Q : g(x) = a} \mu_Q(x, y) 
&= \sum_{b \in P : f(b) = y} \mu_P(a, b).
\end{alignat*}
Thus, Rota's Galois connection theorem is equivalent to the statement $(g_\# \circ \partial_Q)(1_y) = (\partial_P \circ f^\#)(1_y)$ for all $y\in Q$. Since $g_\sharp \circ \partial_Q$ and $\partial_P\circ f^\sharp$ are $\Gr$-module homomorphism, and since $\Gr^Q$ is generated by functions of the form $1_y$, the preceding statement is true if and only if $g_\sharp \circ \partial_Q = \partial_P\circ f^\sharp$.

\subsection{Euler Characteristic and Rota's Theorem}
We now prove in Theorem~\ref{thm:rota_ext} how Theorem~\ref{thm:categorical_RGCT}
decategorifies to Theorem~\ref{thm:rota_second_ver} via the Euler characteristic of the Ext-functor.

The following corollary is an immediate consequence of Theorem~\ref{thm:categorical_RGCT}.
\begin{cor} \label{cor:rota_mobius_co}
Let $f : P \rightleftarrows Q : g$ be a Galois connection, let $N \in \cat^Q$, and let $a \in P$. Then,
\begin{equation*}
\uExt^d(\One_a, f^*N) \cong \uExt^d(g^*\One_a, N).
\end{equation*}
\end{cor}

\begin{thm}\label{thm:rota_ext}
Let \( f : P \rightleftarrows Q : g \) be a Galois connection and let \( N \in \cat^Q \). Then, for \( a \in P \),
\begin{equation}
\label{eq:cor_left}
(\partial_P \circ f^\#)( n)(a) = \chi ( \One_a, f^\ast N ) = \chi ( g^\ast \One_a, N ) = (g_\# \circ \partial_Q)(n) (a),
\end{equation}
where \( n : Q \to \Gr(\cat) \) is the dimension function of \( N \).
\end{thm}

\begin{proof}
We will prove the equalities in equation \eqref{eq:cor_left} one by one.

The middle equality follows directly from Corollary~\ref{cor:rota_mobius_co}.

The function \( f^\# n : P \to \Gr(\cat) \) is the dimension function for \( f^\ast N : Q \to \cat \). Thus, by Theorem~\ref{thm:cohom_euler_char}, we have:
   \[
   (\partial_P \circ f^\# n)(a) = \chi(\One_a, f^\ast N).
   \]

To establish the right equality, consider the spread 
   \[
   Z := \{ x \in Q : g(x) = a \}.
   \]
   By the definition of \( g^\ast \), it follows that \( g^\ast \One_a \cong \One_Z \). Therefore, we can express the Euler characteristic as follows:
   \begin{align*}
\chi(g^*\One_a, N) &= \sum_{d\geq 0}(-1)^d \sum_{\substack{\tau \in \Delta Q \\ \dim \tau = d \\ \min \tau \in Z}} \big[ N(\max \tau) \big] && \text{ by Proposition~\ref{prop:euler_indicator}} \\
&= \sum_{x\in Z}\sum_{d\geq 0}(-1)^d\sum_{\substack{\tau \in \Delta Q \\ \dim \tau = d \\ \min \tau =x}} \big[ N(\max\tau) \big] \\
&= \sum_{x \in Z} \partial_Q n(x) && \text{ by Proposition~\ref{prop:euler_indicator}} \\
&= (g_\# \circ \partial_Q)(n) (a). && \text{ by definition of $g_\#$} \qedhere
\end{align*} 
\end{proof}


\bibliographystyle{alpha}
\bibliography{references}{}

\end{document}